\documentclass[11pt, a4paper,twoside]{article} 
\usepackage[english]{babel} 
\usepackage[utf8]{inputenc} 
\usepackage{vmargin} 
\setmargins{2.5cm}       
{1.5cm}                        
{16.5cm}                      
{23.42cm}                    
{10pt}                           
{1cm}                           
{0pt}                             
{2cm}                           

\usepackage{amsmath}
\usepackage{amssymb}
\usepackage{amsthm}

\usepackage[final]{graphicx}
\usepackage{soul}
\usepackage{comment}
\usepackage{vmargin}
\usepackage{braket}
\usepackage{enumerate}
\usepackage{enumitem} 
\setlist[enumerate,1]  {label={\rm (\roman*)}, leftmargin=1.5em}   
\setlist{noitemsep} 

\usepackage{empheq} 
\include{extsizes}
\usepackage{bigstrut}
\usepackage{float} 
\usepackage{mathrsfs} 
\usepackage{mdframed} 
\usepackage{xcolor} 
{\begin{mdframed}[backgroundcolor=lime]}
	{\end{mdframed}}
{\begin{mdframed}[backgroundcolor=pink]}
	{\end{mdframed}}
\usepackage{version}
\excludeversion{reserva2} 

\usepackage[nottoc]{tocbibind}
\usepackage{amsthm}
\usepackage{wrapfig}
\usepackage{caption2}
\usepackage{esint}
\usepackage[title]{appendix} 
\usepackage{titlesec}
\usepackage{commath}
\usepackage{accents}

\numberwithin{equation}{section}
\newtheorem{theorem}{Theorem}[section]
\newtheorem{proposition}[theorem]{Proposition}
\newtheorem{lemma}[theorem]{Lemma}

\newtheorem{remark}[theorem]{Remark}

\renewenvironment{proof}[1][Proof] {\noindent \textbf{#1.} }
{\  \rule{0.5em}{0.5em}\par \medskip}

\newcommand{\myleq}[1]{\ensuremath{\stackrel{\text{#1}}{\leq}}}

\newcommand{\defeq}{\vcentcolon=}
\newcommand{\eqdef}{=\vcentcolon}

\newcommand{\be}{\begin{equation} \label}
\newcommand\ee{\end{equation}}

\newcounter{assumption}

\renewcommand{\theassumption}{A\arabic{assumption}}

\renewcommand{\abs}[1]{\left| #1 \right|}

\newcommand{\eps}{\varepsilon}

\newcommand*\Lap{\mathop{}\!\mathbin\Delta}

\titleformat{\paragraph}
{\normalfont\normalsize\bfseries}{\theparagraph}{1em}{}
\titlespacing*{\paragraph}
{0pt}{3.25ex plus 1ex minus .2ex}{1.5ex plus .2ex}

\newcommand*\adh[1]{\overline{#1}} 

\newcommand{\Rn}{{\mathbb{R}^n}}
\newcommand{\R}{\mathbb{R}}

\def\Xint#1{\mathchoice
	{\XXint\displaystyle\textstyle{#1}}%
	{\XXint\textstyle\scriptstyle{#1}}%
	{\XXint\scriptstyle\scriptscriptstyle{#1}}%
	{\XXint\scriptscriptstyle\scriptscriptstyle{#1}}%
	\!\int}
\def\XXint#1#2#3{{\setbox0=\hbox{$#1{#2#3}{\int}$ }
		\vcenter{\hbox{$#2#3$ }}\kern-.580\wd0}}

\def\dashint{\Xint-}
\renewcommand{\fint}{\dashint}   

\usepackage{fancyhdr}

\usepackage{mathtools}

\usepackage[pagebackref]{hyperref}
\hypersetup{
	colorlinks=true,
	citecolor = blue, 
	linkcolor= blue,
}

\DeclareMathOperator*{\esssup}{ess\,sup}

\begin{document}

\title{Global existence and asymptotic behavior for diffusive Hamilton-Jacobi equations with Neumann boundary conditions}

\author{Joaquín Domínguez-de-Tena \thanks{Partially supported
		by Projects PID2019-103860GB-I00 and  PID2022-137074NB-I00,  MICINN and   GR58/08
		Grupo 920894, UCM, Spain and by ICMAT Severo Ochoa
		Grant CEX2019-000904-S funded by MCIN/AEI/ 10.13039/501100011033.}\ ${}^{,1}$ \\
	Philippe Souplet ${}^{2}$}

\date{}
\maketitle

\setcounter{footnote}{2}
\begin{center}
	\small ${}^{1}$Universidad
	Complutense de Madrid \\ Departamento de Análisis Matemático y  Matemática Aplicada,  \\ and \\
	Instituto de Ciencias Matemáticas CSIC-UAM-UC3M-UCM	
	\\ \ \\
	${}^{2}$
	Université Sorbonne Paris Nord, CNRS UMR 7539, LAGA, Villetaneuse, France \\
\end{center}

\makeatletter
\begin{center}
	${}^{1}${E-mail:
		joadomin@ucm.es}
	\\ 
	${}^{2}${E-mail:
		souplet@math.univ-paris13.fr}
\end{center}

\begin{center}
{\it This work is dedicated to the memory of Marek Fila, dear friend and colleague}
\end{center}
\makeatother

\begin{abstract}
	We investigate the diffusive Hamilton-Jacobi equation
	$$
	u_t-\Lap u = |\nabla u|^p
	$$
	with $p>1$, in a smooth bounded domain of 
	$\Rn$ with homogeneous Neumann boundary conditions and $W^{1,\infty}$ initial data.
	We show that all solutions exist globally, are bounded and converge in $W^{1,\infty}$ norm to a constant as $t\to\infty$,
	with a uniform exponential rate of convergence given by the second Neumann eigenvalue.
	This improves previously known results,
	which provided only an upper polynomial bound on the rate of convergence and required the convexity of the domain.
	Furthermore, we extend these results to a rather large class of nonlinearities $F(\nabla u)$ instead of~$|\nabla u|^p$.
\end{abstract}
\newpage
\section{Introduction and main results}
\subsection{Background and motivation}
In this paper, we study the Neumann problem for the diffusive Hamilton-Jacobi equation

\begin{equation}
	\label{eqn:HJpar}
	\left\{
	\begin{aligned}
		u_t-\Lap u & = |\nabla u|^p,&&x\in\Omega,\ t>0, \\
		\frac{\partial u}{\partial\nu} & = 0,&&x\in\partial \Omega,\ t>0, \\
		u(x,0) & =u_0(x),&&x\in\Omega.
	\end{aligned}
	\right.
\end{equation}
Throughout this paper we assume that $p>1$ and $\Omega\subset \Rn$ is a 
bounded $C^{2+\alpha}$ domain for some $0<\alpha<1$, $\frac{\partial u}{\partial\nu}$ denotes the exterior normal derivative,
 and we shall denote by $(S(t))_{t\ge 0}$ the Neumann heat semigroup. Also we set
$$X:=W^{1,\infty}(\Omega),
\qquad \|v\|_X=\|v\|_{L^\infty(\Omega)}+\|\nabla v\|_{L^\infty(\Omega)}.$$

For any $u_0\in X$,
problem \eqref{eqn:HJpar} admits a unique, maximal classical solution 
			\begin{equation}
	\label{regulu0}
 u\in C^{2,1}(\adh{\Omega}\times(0,T)),\quad\hbox{with $\lim_{t\to 0}\|u(t)-S(t)u_0\|_X=0$}
	\end{equation}
(see Appendix for a proof),  and we also have
		\be{regulu1b}
		u\in L^\infty_{loc}([0,T);X).
			\ee
In what follows, we shall denote this solution by $u$ and by
$T=T(u_0)\in(0,\infty]$ its maximal existence time.
Moreover, $u$ satisfies the  $X$-blowup alternative:
	\begin{equation}
		\label{eqn:cont}
		T<\infty\ \Longrightarrow\ \lim_{t\to T}\|u(t)\|_X=\infty.
			\end{equation}

This  diffusive Hamilton-Jacobi equation has been widely studied in $\Rn$ or in $\Omega$ with Dirichlet boundary conditions.
In the case of $\Rn$, all solutions are global and their asymptotic behavior depends on the relation between $p$ and $n$ and on the 
asymptotic properties of the initial data (see \cite{AB, BeLa, BSW1, BSW2, BKL, GGK, LaSo, Gil, So07}).
As for the Dirichlet problem, all solutions exist globally if $p\le 2$, whereas 
for $p > 2$ and suitably large initial data, the local classical solution develops singularities in finite time
(see \cite{Al, ABG, So2, HM04, BDL}),
 although solutions may be continued as a global, generalized viscosity solution (see \cite{BaD}). 
These singularities are of gradient blowup type, the function $u$ itself remaining bounded, and are located on some part of the boundary.
Numerous results are available on the gradient blowup sets \cite{SZ, LiSo, Est},
rates \cite{CG, SV, GH, ZL, PS3, QS, AtSo2, MiSo2},
profiles \cite{ARS, CG, SZ, PS1, PS3,  FPS19, MiSo2}
and properties of the weak continuation after blowup \cite{FL, PZ, PS2, QR, MiSo1}.
Other interesting issues have also been investigated, such as
controllability \cite{PZ}, maximal regularity \cite{CiG} and ergodicity 
(see, e.g.,~\cite{SZ, TT, BPT, BQR} and references therein).
We refer \cite{Dl, Ku, FL, Gi, AF, AI, FTW,  At, ZH, ZL2, ZL3, AtBa, LYZ, AtSo1, FLa}
 for gradient blowup studies on the Dirichlet problem for more general, semilinear or quasilinear parabolic equations with nonlinearities involving the gradient.

 On the other hand, although both the Dirichlet and the Neumann problems are important in the context of stochastic control problems
(see \cite{LaLi, BDL, FlSo}), 
there are only few studies for the case of the Neumann problem. 
For $1<p<2$, it was shown in \cite{Dab2} that $u$ exists globally.
Then in \cite{BDab2}, for all $p>1$ but under the restriction that $\Omega$ is convex, it was proved that 
all solutions of \eqref{eqn:HJpar} are global and bounded, and that they 
converge in $C^1$ to a constant as $t\to\infty$, with an upper polynomial bound on the rate of convergence. Namely:
$$\|u(t)-c\|_{C^1}\le C(t+1)^{-\gamma},\quad t\ge 0,$$
with $\gamma=\gamma(n,p)>0$, for some $c\in\R$, $C>0$.
This indicates in particular that the singularity formation phenomena found in the Dirichlet case
for~$p>2$ do not occur for the Neumann problem, at least in convex domains.

In this paper we are interested in the question of global existence, boundedness and asymptotic behavior as $t\to\infty$
for problem \eqref{eqn:HJpar}.
Our results are strongly motivated by the work  \cite{BDab2} and will improve the results therein in two ways:
\vskip 1pt

$\bullet$ We prove {\it exponential} instead of polynomial convergence to a constant;
\vskip 1pt

$\bullet$  We remove the {\it convexity} assumption on $\Omega$ (both for global existence and for convergence).

\vskip 1pt	 
\noindent Moreover, the decay estimate is uniform for all times and with respect to the initial data,
provided $\|\nabla u_0\|_\infty$ remains bounded.
Furthermore, we shall show that such properties continue to hold for a rather large class of gradient nonlinearities. 

\subsection{Main results}
In the rest of the paper,
$$\hbox{$\lambda>0$ denotes
the second eigenvalue of $-\Delta$ with Neumann boundary conditions in $\Omega$.}$$
Our main result on problem \eqref{eqn:HJpar} is the following.

\begin{theorem}
	\label{thm:final}
	Consider problem \eqref{eqn:HJpar} with $p>1$ and $u_0\in X$.
Then 
	\begin{equation}
		\label{eqn:exprate0}
		T=\infty, \quad \displaystyle\sup_{t\ge 0}\|u(t)\|_X<\infty
			\end{equation}
		and there exist constants $c\in\R$, $C>0$ such that
	\begin{equation}
		\label{eqn:exprate}
		\norm{u(t)-c}_X\leq C e^{-\lambda t}, \qquad t\geq 0.
	\end{equation}
	Moreover the constant $C$ is uniform for $\|\nabla u_0\|_\infty$ bounded.
\end{theorem}

Next consider the more general problem 
	\begin{equation}
		\label{eqn:HJ}
		\left\{
		\begin{aligned}
			u_t-\Lap u & = F(\nabla u),&&x\in \Omega,\ t>0, \\
			\frac{\partial u}{\partial\nu} & = 0,&&x\in\partial \Omega,\ t>0, \\
			u(x,0) & =u_0(x),&&x\in \Omega,
		\end{aligned}
		\right.
	\end{equation}
where 
	\begin{equation}
		\label{eqn:ass0}
		F\in C^1(\Rn;\R),\quad F(0)=0.
			\end{equation}
The local existence-uniqueness result recalled above as well as property \eqref{eqn:cont}, remain valid for problem \eqref{eqn:HJ}
(see Appendix).
We still denote by $u$ and $T$ the maximal classical solution of \eqref{eqn:HJ} and its maximal existence time,
without risk of confusion.
Concerning the nonlinearity $F$, we shall assume:
	\begin{equation}
		\label{eqn:ass4}
\lim_{|z|\to\infty}\frac{|F(z)|}{|z|\sqrt{1+|\nabla F(z)|}}=\infty;
	\end{equation}
		\begin{equation}
		\label{eqn:ass3}
	\hbox{there exist $p>1$ and $\delta>0$ such that 
		$|F(z)|\leq \abs{z}^p$ for all $\abs{z}\leq \delta.$}
	\end{equation}

Our main result on problem \eqref{eqn:HJ} is the following theorem
	(of which Theorem \ref{thm:final} is immediately seen to be a special case).
		
	\begin{theorem}
	\label{thm:finalgen}
	Assume \eqref{eqn:ass0}, \eqref{eqn:ass4}, \eqref{eqn:ass3} and consider problem \eqref{eqn:HJ} with $u_0\in X$.
	\smallskip
	
(i) Then the conclusions  \eqref{eqn:exprate0} and  \eqref{eqn:exprate} of Theorem \ref{thm:final} remain valid.
	\smallskip
	
(ii) Assume in addition that, for some $q>1$ and $c_0>0$,
	\begin{equation}
		\label{eqn:ass5}
		F(z)\ge c_0 |z|^q,\quad z\in\Rn.
		\end{equation}	
	Then the constant $C$ in \eqref{eqn:exprate} is uniform for $\|\nabla u_0\|_\infty$ bounded.
	\end{theorem}

	\medskip
	\begin{remark}
		(i) The decay rate \eqref{eqn:exprate} is optimal in general.
	 Indeed, for any $F$ satisfying  \eqref{eqn:ass0}, \eqref{eqn:ass3}, there exist arbitrarily small initial data $u_0\in X$
	 such that
	 $u$ is global and satisfies
	 $$C_1e^{-\lambda t}\le \|u(t)\|_{H^1(\Omega)}\le C_2e^{-\lambda t},\quad t\ge 0$$
	 for some $C_2>C_1>0$ (see \cite[Theorem~51.21]{QS}).
	 Such initial data are found on the stable manifold of the equilibrium $0$, starting close to the direction of the eigenfunction associated with the eigenvalue $\lambda$.
	 On the other hand, faster exponential decay rates may also be obtained by considering higher eigenvalues (see \cite[Theorem~51.21]{QS}).

		\smallskip
		
		(ii) Theorem \ref{thm:finalgen} holds for instance for $F(\nabla u)=\exp(|\nabla u|^q)-1$ with $q>1$.
	We refer to, e.g.,~\cite{FL, FTW, ZH, ZL2, ZL3} for results on the Dirichlet problem with exponential nonlinearities.

		\smallskip

	(iii) The constant $c$ in \eqref{eqn:exprate} is given by the (implicit) relation 
	\be{defc}
	c=\fint_\Omega u_0 + \int_0^\infty \fint_\Omega F(\nabla u(t))dt,
	\ee
	where $\fint$ stands for the average integral (see Proposition~\ref{prop:finalgen3}).
	
		\smallskip
		
	(iv) For problem \eqref{eqn:HJ}, the assumption $F(0)=0$ in \eqref{eqn:ass0} can be removed by means of the simple transformation
	$v=u-F(0)t$, so that $v$ solves the same problem with $F(z)$ replaced by $\tilde F(z)=F(z)-F(0)$.
	Also, in view of the transformation $v=-u$, assumption \eqref{eqn:ass5} can be replaced with $F(z)\le -c_0 |z|^q$.
	\end{remark}

	\begin{remark}\label{locex}
	Local existence-uniqueness for low regularity initial data is out of the scope of this paper.
	For $1<p<2$, this was studied in (optimal) $L^q$ spaces and in spaces of measures for the Cauchy, Dirichlet and Neumann problems;
	see \cite{CLS, An, BSW1, BSW2, BDab, BVD, Dab2}.
	For all $p>1$, local existence-uniqueness was proved for bounded continuous initial data in \cite{GGK, BaD}
	for the Cauchy and Dirichlet problems, and in \cite{BDab2} for the Neumann problem under the restriction that $\Omega$ is convex.
	Regularization estimates as $t\to 0^+$ can also be found in the above works.
		\end{remark}

\subsection{Organization of the paper and main ideas of proofs}

	\medskip
	
 In Section~\ref{sec:globalex}, we establish global existence and $W^{1,\infty}$ boundedness under assumptions
\eqref{eqn:ass0} and \eqref{eqn:ass4} on $F$ (cf.~Proposition~\ref{prop:finalgen}).
 To this end, after deriving some first maximum principle bounds on $u$ and $u_t$
(Proposition~\ref{thm:uac}), we proceed to estimate $|\nabla u|$.
 This is done by a Bernstein technique, using a parabolic equation for a quantity involving $\phi=|\nabla u|^2$. 
If $\Omega$ is convex, then $\phi$ just satisfies the boundary condition $\frac{\partial \phi}{\partial\nu}\le 0$, 
a fact which is crucially used in \cite{BDab2}.
But for general domains, we only have a Robin type condition $\frac{\partial \phi}{\partial\nu}\le K\phi$,
so that we can no longer directly apply the maximum principle to the $\phi$-equation to bound $\nabla u$.
 To cope with this difficulty, some new ideas are needed.
 We first introduce a multiplicative perturbation of $|\nabla u|^2$, replacing $\phi$ by $h=\phi\psi$ for a well-chosen
auxiliary function $\psi>0$.
This change of variable transfers the \textit{bad term} on the boundary to a \textit{bad term} in the interior of the domain (see \eqref{eqn:globalexeq8}). This term can be compensated by means of 
 the absorbing Hessian term $|D^2u|^2$ arising from the Bernstein device, combined with the 
previously obtained uniform estimate of $u_t$.
We note that, when $\Omega$ is convex, different Bernstein-type arguments are also used in \cite{BDab2}
and allow to obtain regularization estimates as $t\to 0^+$ (an issue which, as mentioned in Remark~\ref{locex}, 
is out of the scope of the present paper).
	\smallskip
	
 In Section~\ref{sec:globalex2}, for all global bounded solutions, under the sole assumption
\eqref{eqn:ass0}, we prove convergence (without rate) to a constant 
(cf.~Proposition~\ref{prop:finalgen2}).
 Our proof relies on modifications of ideas from \cite{Dab1, SZ}.
It is based on dynamical systems and invariance principle arguments, using the quantity
$M(t):=\max_{\overline\Omega}u(t)$ as a Liapunov functional, 
the strict Liapunov property being deduced from the strong maximum principle.
 
	\smallskip
 In Section~\ref{sec:conv}, under assumptions \eqref{eqn:ass0} and \eqref{eqn:ass3} on $F$
we establish a global existence and exponential decay result for small initial data
(cf.~Proposition~\ref{prop:finalgen3}). 
The proof is done
by means of semigroup arguments, using suitable gradient estimates on the Neumann semigroup.

	\smallskip
	
 In Section~\ref{sec:conv2}, under the additional assumption \eqref{eqn:ass5}, we show the uniformity of the constant in  \eqref{eqn:exprate}.
 The proof is based on regularity estimates, combined with semigroup and
smoothing arguments.
Theorem \ref{thm:finalgen} (hence Theorem \ref{thm:final}) will then follow by combining the results of Sections~\ref{sec:globalex}--\ref{sec:conv2}.

	\smallskip
	Finally, in appendix, we give a version of the comparison principle suitable to our needs
	and a precise statement of local existence-uniqueness for problem \eqref{eqn:HJ} for initial data $u_0\in X$.
		Although the result is essentially standard, we provide a proof for the convenience of the readers.

\section{Global existence and boundedness}
\label{sec:globalex}

In this section, we establish global existence and $C^1$ boundedness under assumptions
\eqref{eqn:ass0} and~\eqref{eqn:ass4} on $F$.

\begin{proposition}
	\label{prop:finalgen}
	Assume \eqref{eqn:ass0}, \eqref{eqn:ass4} and consider problem \eqref{eqn:HJ} with $u_0\in X$.
	Then 
	\begin{equation}
	\label{prop:finalgen-global}
	\hbox{$T=\infty$ \quad and \quad  $\displaystyle\sup_{t\ge 0}\|u(t)\|_X<\infty$.}
		\end{equation}
	\end{proposition}
	
	\begin{remark} 
	More precisely, the proof of Proposition~\ref{prop:finalgen} gives the estimate
			\begin{equation}
	\label{barh2}
		\displaystyle\sup_{t\ge t_0}\|\nabla u(t)\|_\infty
		\leq C\bigl(\Omega, F,\|\nabla u(t_0)\|_\infty,\|u_t(t_0)\|_{L^\infty(\Omega)}\bigr),\quad t\in [t_0,T),
		\end{equation}
		for any $t_0\in(0,\infty)$,
		where $C$ remains bounded for bounded values of its last two arguments
		(note that the last one is finite owing to the regularity \eqref{regulu0} of $u$).
		
			\end{remark}
			
We first establish some useful monotonicity properties of $u$ and $u_t$.
	
	\begin{proposition}
		\label{thm:uac}
			Assume \eqref{eqn:ass0}, consider problem \eqref{eqn:HJ} with $u_0\in X$, and set
			$$m(t)=\min_{x\in\overline\Omega} u(t,x),\quad M(t)=\max_{x\in\overline\Omega} u(t,x),\quad
			L(t)=\norm{u_t}_{L^\infty(\Omega)}.$$
	 Then
	 			\begin{equation}
				\label{eqn:monMm}
	 \hbox{$M(t)$ is nonincreasing and $m(t)$ is nondecreasing on $[0,T)$,}
			\end{equation} 
	 and
	 			\begin{equation}
				\label{eqn:monut}
	\hbox{$L(t)$ is nonincreasing on $(0,T)$.}
			\end{equation} 
		In particular
		\begin{equation}
						\label{eqn:bounduu}
			\norm{u(t)}_{L^\infty(\Omega)}\leq \norm{u_0}_{L^\infty(\Omega)}, \quad t\in (0,T).
		\end{equation}
	\end{proposition}
	
\begin{proof}
	Let $t_0\geq0$. We use  the comparison principle, Proposition~\ref{thm:comp}, with the function $u(t_0+\cdot)$ 
and the constant function $M(t_0)$ as a supersolution. It follows that $u(t_0+t,x)\le M(t_0)$ on $\Omega\times(0,T-t_0)$, so that $M(s)\le M(t_0)$ for all $s\in(t,T)$, hence the first part of 
	\eqref{eqn:monMm} follows.
	The second part is similar.
	\smallskip
	
	To prove \eqref{eqn:monut}, fix $t_0\in(0,T)$ and let $\ell=\norm{u_t(t_0)}_{L^\infty(\Omega)}$. 
	Considering $v(t)=u(t_0+t)-\ell t$, we have
		\begin{equation}
		\label{eqn:thmuaceq1}
	\left\{
	\begin{aligned}
		v_t-\Lap v & = F(\nabla v) - \ell,&&x\in \Omega,\ 0<t<T-t_0,\\
		v(0) & = u(t_0), \qquad &&x\in\Omega, \\
		\frac{\partial v}{\partial\nu} & = 0,&&x\in\partial \Omega,\ 0<t<T-t_0.
	\end{aligned}
	\right.
	\end{equation}	
	It is easy to check that the function $\bar v(t,x)=u(t_0,x)$ is a supersolution of \eqref{eqn:thmuaceq1}.
	We deduce from Proposition~\ref{thm:comp} that $v\leq u(t_0)$, hence
	\begin{equation}
		\label{eqn:thmuaceq1b}
		u(t_0+t)\leq u(t_0)+\ell t,\quad x\in \Omega,\ 0<t<T-t_0.
		\end{equation}	
	Now, given $h\in (0,T-t_0)$, we consider $z(t)=u(t+h)-\ell h$. Using \eqref{eqn:thmuaceq1b}, we see that
	$$
	\left\{
	\begin{aligned}
		z_t-\Lap z & = F(\nabla z) && x\in\Omega,\ t_0<t<T-h\\
		z(t_0) & \leq u(t_0) \qquad && x\in\Omega,\\
		\frac{\partial z}{\partial\nu} & = 0 && x\in\partial \Omega,\ t_0<t<T-h. \\
	\end{aligned}
	\right.
	$$
	Using the comparison principle Proposition~\ref{thm:comp} again, we obtain $z\leq u$ 
	in $\Omega\times(t_0,T-h)$. Consequently, 
	$\frac{u(t+h)-u(t)}{h}\leq \ell$ for $t_0<t<T-h$. 
	Letting $h\to 0^+$, 
	we obtain $u_t(t)\leq \ell$ for every $t\in(t_0,T-h)$. Arguing analogously one obtains $u_t(t)\geq - \ell$. 
		This proves \eqref{eqn:monut}.
\end{proof}

		\begin{proof} [Proof of Proposition~\ref{prop:finalgen}]
			\smallskip
	
	{\bf Step 1.} {\it Regularity.}
	 For each $i\in\{1,\dots,n\}$, $z:=\partial_{x_i}u$ is a distributional solution of $z_t-\Delta z=g:=(\nabla F)(\nabla u)\cdot\nabla z$
	in $Q_T$, and we have $g\in L^\infty_{loc}(\overline\Omega\times(0,T))$, owing to $u\in C^{2,1}(\overline\Omega\times(0,T))$ and $F\in C^1$.
		By parabolic regularity (see, e.g.,~\cite[Theorem 48.1 and Remark~48.3(i)]{QS}), it follows that
			\begin{equation}
	\label{regulu1}
	\nabla u\in W^{2,1;m}_{loc}(\Omega\times(0,T)),\quad 1<m<\infty.
	\end{equation}
		Set $\phi=|\nabla u|^2$. By \eqref{regulu0} and \eqref{regulu1}, we note that
			\begin{equation}
	\label{regulu2}
	\phi\in C([0,T);L^m(\Omega))\cap W^{2,1;m}_{loc}(\Omega\times(0,T)),\quad 1<m<\infty,
		\end{equation}
		and
			\begin{equation}
	\label{regulu3}
	\nabla \phi\in C(\overline\Omega\times(0,T)).
		\end{equation}

				{\bf Step 2.} {\it Equation for $|\nabla u|^2$.}
				We claim that
		\begin{equation}
	\label{eqn:globalexeq4}
	\left\{
	\begin{aligned}
		\phi_t-\Lap \phi & = b\cdot \nabla\phi-2|D^2u|^2 \\
		\frac{\partial \phi}{\partial\nu}&\leq K\phi	
	\end{aligned}\right. 
\end{equation}
for some $K=K(\Omega)>0$, where $b=(\nabla F)(\nabla u)$. 

\smallskip

		Taking gradients in \eqref{eqn:HJ}, we have
		\begin{equation}
			\label{eqn:globalexeq2}
			(\nabla u)_t-\nabla(\Lap u)= \nabla (F(\nabla u)).
		\end{equation}
	We use the Bochner identity
$$
		\frac{1}{2}\Lap (|\nabla u|^2)=\nabla u \cdot \nabla (\Lap u)+|D^2u|^2
$$
where $|D^2u|^2=\sum_{i,j}\bigl|\frac{\partial^2 u}{\partial x_i \partial x_j}\bigr|^2$.
Multiplying \eqref{eqn:globalexeq2} by $2\nabla u$ and combining with the last equation, we obtain
$$
\frac{\partial}{\partial t}|\nabla u|^2-\Lap |\nabla u|^2=2\nabla (F(\nabla u))\cdot\nabla u-2|D^2u|^2.
$$
Since $2\nabla (F(\nabla u))\cdot\nabla u = 2\sum_{i,j=1}^n\frac{\partial F}{\partial x_j}\frac{\partial^2 u}{\partial x_i \partial x_j}\frac{\partial u}{\partial x_i}=(\nabla F)(\nabla u)\cdot \nabla(|\nabla u|^2)$, we get the PDE in \eqref{eqn:globalexeq4}.
On the other hand, since $u_\nu=0$ on $\partial\Omega$, it is known (see, e.g.,~\cite[Lemma 4.2]{MiSo0}),
that there exists $K=K(\Omega)>0$ such that
		\begin{equation}
			\label{eqn:RobinBC}
			\frac{\partial}{\partial\nu}|\nabla u|^2\leq  K|\nabla u|^2\hbox{ on $\partial \Omega$.}
				\end{equation}

	{\bf Step 3.} {\it Multiplicative perturbation of $|\nabla u|^2$ and conclusion.}
Now, because of the Robin boundary conditions in \eqref{eqn:RobinBC}, we cannot directly apply the maximum principle to bound $|\nabla u|^2$,
unlike in \cite{BDab2} (where $K=0$ owing to the convexity of $\Omega$).
To overcome this, we introduce the problem
	\begin{equation}
	\label{AuxRobin}
		\left\{
		\begin{aligned}
			-\Lap \psi& = 1,\qquad && x\in\Omega\\
			\frac{\partial \psi}{\partial\nu}& =-K\psi,&& x\in\partial \Omega.
		\end{aligned}\right.
	\end{equation}
		The existence of a unique solution $\psi\in  C^2(\overline\Omega)$ of \eqref{AuxRobin}
		 is classical; see,~e.g., \cite[Theorem 6.31]{GT}.
Moreover, we have
$\sigma=\min_{\overline\Omega}\psi>0$.
Indeed, if $x_0\in \overline\Omega$ is such that $\psi(x_0)=\sigma$,
we necessarily have $x_0\in\partial\Omega$
(since otherwise $1=-\Lap \psi(x_0)\le 0$, a contradiction) and,
by Hopf's lemma, it thus follows that $-K\sigma=\frac{\partial \psi}{\partial\nu}(x_0)<0$. 

We now define $h=\psi \phi$. We have
	\begin{equation}
	\label{hNeumann}
\frac{\partial h}{\partial \nu}=\psi\frac{\partial \phi}{\partial \nu}+\phi\frac{\partial \psi}{\partial \nu}\leq K\phi\psi-K\phi\psi=0,
\quad x\in\partial\Omega,\ 0<t<T.
	\end{equation}
Let us compute the equation satisfied by $h$.
Using \eqref{eqn:globalexeq4} and 
$\nabla \phi \cdot \nabla \psi = \nabla h \cdot\frac{\nabla \psi}{\psi}-\frac{\abs{\nabla \psi}^2}{\psi^2}h$, we get
$$\begin{aligned}
	h_t-\Lap h 
	&= \psi(\phi_t-\Lap \phi)-2\nabla\phi\cdot\nabla\psi-\phi\Lap \psi \\
	&= \psi\left(b\cdot\nabla \phi-2|D^2u|^2\right)-2\frac{\nabla \psi}{\psi}\nabla h+2\frac{\abs{\nabla\psi}^2}{\psi^2}h-\phi\Lap \psi
\end{aligned}$$
hence, since $\psi (b\cdot \nabla \phi)=b\cdot \nabla h - \phi (b\cdot \nabla\psi)$,
\begin{eqnarray}
	h_t-\Lap h 
	&=& \Big(b-2\frac{\nabla \psi}{\psi}\Big)\cdot\nabla h+\phi\left(-\Lap \psi-b\cdot\nabla \psi\right)-2\psi|D^2u|^2+2\frac{\abs{\nabla\psi}^2}{\psi^2}h \notag\\ 
	&=& \Big(b-2\frac{\nabla \psi}{\psi}\Big)\cdot\nabla h-2\psi|D^2u|^2
+\Big(-\frac{\Lap \psi}{\psi}-b\cdot\frac{\nabla \psi}{\psi}+2\frac{\abs{\nabla\psi}^2}{\psi^2}\Big)h. \label{eqn:globalexeq8}
\end{eqnarray}
Pick any $t_0\in(0,T)$. By Proposition~\ref{thm:uac}, we have
$$\sup_{t\in[t_0,T)} \norm{u_t(t)}_{L^\infty(\Omega)}\le M_1:=\|u_t(t_0)\|_{L^\infty(\Omega)},$$
hence
$$
	2\psi|D^2u|^2\geq \frac{2\psi}{n}|\Lap u|^2=\frac{2\psi}{n}\left(F(\nabla u)-u_t\right)^2\ge c_1 |F(\nabla u)|^2-C_1M_1^2,
	\quad (x,t)\in Q_0:=\Omega\times (t_0,T),
$$
for some $c_1,C_1>0$ depending only on $\Omega$ (through the function $\psi$,
since $\psi\in C^2(\adh{\Omega})$ and $\psi\geq \sigma>0$).
Setting $b_1=2\frac{\nabla \psi}{\psi}-b$, it follows that
$$
	h_t-\Lap h + b_1\cdot\nabla h\le C_2\bigl(1+M_1^2+|\nabla F(\nabla u)|\bigr)|\nabla u|^2-c_1 |F(\nabla u)|^2, 
	\quad a.e.~(x,t)\in Q_0,
$$
for some $C_2>0$ depending only on $\Omega$ (again through $\psi$).
By assumption \eqref{eqn:ass4}, there exists $C_3=C_3(\Omega, F,M_1)>0$ such that
$$c_1|F(z)|^2\ge (1+C_2)\bigl(1+M_1^2+|\nabla F(z)|\bigr)z^2-C_3,\quad z\in\Rn.$$
We deduce that
$$	h_t-\Lap h + b_1\cdot\nabla h\le -\abs{\nabla u}^2+C_3\le -\sigma h+C_3, 	\quad a.e.~(x,t)\in Q_0.$$
where $b_1\in L^\infty_{loc}([0,T);L^\infty(\Omega))$ owing to \eqref{regulu1b}.  Therefore,
 $\bar h:=h-K$ with $K=\max(C_3/\sigma,\|h(t_0)\|_\infty)$ satisfies
$$	\bar h_t-\Lap \bar h \le C|\nabla \bar h|, 	\quad a.e.~(x,t)\in Q_0\cap\{\bar h\ge 0\}$$
 and $\bar{h}(t_0)\leq 0$. Since $\bar h$ enjoys the same regularity \eqref{regulu2}-\eqref{regulu3} as $\phi$, 
we may apply the maximum principle in Proposition~\ref{thm:comp0}, which yields $\bar h\le 0$ in $Q_0$.
Recalling \eqref{eqn:cont} and \eqref{eqn:bounduu}, this completes the proof
(and shows estimate \eqref{barh2}).
 \end{proof}

	\section{Convergence to a constant for global bounded solutions}
\label{sec:globalex2}

 In this section, under the sole assumption
\eqref{eqn:ass0}, we prove convergence (without rate) of any global bounded solution to a constant.

	\begin{proposition}
	\label{prop:finalgen2}
	Assume \eqref{eqn:ass0} and consider problem \eqref{eqn:HJ} with $u_0\in X$.
	If \eqref{prop:finalgen-global} holds, then there exists a constant $c\in\R$ such that
$$\lim_{t\to \infty} \norm{u(t)-c}_X = 0.$$
		\end{proposition}

	\begin{proof} We first observe that, for every $t_0>0$ and $\alpha\in(0,1)$, we have 
	\begin{equation}
	\label{Calpha-bound}
		\norm{u}_{C^{2+\alpha,1+\alpha/2}(\adh{\Omega}\times (t_0,\infty))}<\infty.
	\end{equation}
Indeed, fix a function $\theta\in C^1(\R)$ such that $\theta(s)=0$ for $s\le 0$ and $\theta(s)=1$ for $s\ge t_0$.
For given $\tau\ge 0$, the function $\tilde u_\tau(x,t)=\theta(t-\tau)u(x,t)$ satisfies
	\begin{equation}
	\label{Calpha-bound2}
	\partial_t\tilde u_\tau-\Lap \tilde u_\tau= f_\tau:=\theta_t(t-\tau)u+\theta(t-\tau)F(\nabla u).
		\end{equation}
For each $q\in(1,\infty)$, since 
$\sup_{\tau\ge 0}\|f_\tau\|_{L^q(\Omega\times (\tau,\tau+t_0+1))}<\infty$
and $\tilde u_\tau(\tau)=0$,
it follows from $L^q$ parabolic estimates  (see, e.g.,~\cite[Theorem 48.1]{QS}) that 
$$\sup_{\tau\ge 0}\norm{\tilde u_\tau}_{W^{2,1;q}(\Omega\times (\tau,\tau+t_0+1))}<\infty$$
 hence,
by standard embeddings, 
$\sup_{\tau\ge 0}\norm{\tilde u_\tau}_{C^{1+\alpha,\alpha/2}(\adh{\Omega}\times [\tau,\tau+t_0+1])}<\infty$.
Consequently, 
$$\norm{u,\nabla u}_{C^{\alpha,\alpha/2}(\adh{\Omega}\times [t_0,\infty))}<\infty,$$
for each $t_0>0$.
Going back to \eqref{Calpha-bound2} and applying Schauder parabolic estimates (see, e.g.,~\cite[Theorem 48.2]{QS}), 
we deduce \eqref{Calpha-bound}.

On the other hand, by Proposition~\ref{thm:uac}, $M(t)=\max_{\overline\Omega}{u(t)}$ is non-increasing in $t$ and bounded, hence
	\begin{equation}
	\label{ConvSup}
	\exists \ell:=\lim_{t\to\infty} M(t).
			\end{equation}
	Now pick any sequence of times $t_j\to \infty$ and define
	$u_j(x,t)=u(x,t_j+t)$, which satisfies
$$
	\left\{
	\begin{aligned}
		(u_j)_t-\Lap u_j & = F(\nabla u_j) \qquad && \hbox{in }\Omega\times(0,\infty) \\
		\frac{\partial u_j}{\partial\nu} & = 0 \qquad && \hbox{on }\partial \Omega \times(0,\infty).
	\end{aligned}
	\right.
$$
By \eqref{Calpha-bound},  the functions $u_j$ are uniformly bounded in $C^{2+\alpha,1+\alpha/2}(\Omega\times (0,\infty))$. 
Therefore, by the Ascoli-Arzela Theorem, we obtain a subsequence (not relabeled) $u_j\to z$ in $C^{2,1}_{loc}(\adh{\Omega}\times[0,\infty))$,
hence $z$ satisfies
\begin{equation}
\label{eqzz}
	\left\{
	\begin{aligned}
		z_t-\Lap z & = F(\nabla z) \qquad && \hbox{in } \Omega\times(0,\infty) \\
		\frac{\partial z}{\partial\nu} & = 0 \qquad && \hbox{on }\partial \Omega \times(0,\infty).
	\end{aligned}
\right.
\end{equation}
Moreover, by \eqref{ConvSup}, we have
$\max_{\overline\Omega}z(t)=\ell$ for all $t\geq 0$.
Since $\bar z:=\ell$ is also a solution of \eqref{eqzz}, it follows from the (strong) comparison principle in
Proposition~\ref{thm:comp} that $z\equiv\ell$.
As the limit is independent of the subsequence, the whole sequence converges to the constant $\ell$, and we obtain the desired result.
\end{proof}

\section{Global existence and exponential decay of the gradient for small initial data}
\label{sec:conv}

 In this section, under assumptions \eqref{eqn:ass0} and \eqref{eqn:ass3} on $F$,
we establish global existence and exponential decay for small initial data.

	\begin{proposition}
	\label{prop:finalgen3}
	Assume \eqref{eqn:ass0}, \eqref{eqn:ass3} and consider problem \eqref{eqn:HJ} with $u_0\in X$.
	There exist constants $\eps_0=\eps_0(\Omega,p,\delta)>0$ and $C=C(\Omega)>0$ with the following property.
	If $\|\nabla u_0\|_\infty\le \eps_0$, then
	there exist constants $c\in\R$  and $C>0$  such that
$$		\norm{u(t)-c}_X\leq C\|\nabla u_0\|_\infty e^{-\lambda t}, \quad t\geq 0.$$
Moreover, $c$ is given by formula \eqref{defc}.
\end{proposition}

 In view of the proofs of   Propositions~\ref{prop:finalgen3} and \ref{prop:finalgen4}, we first collect some relevant linear estimates on the Neumann heat semigroup $S(t)$ on $\Omega$.
	A proof is given in Appendix.
	
	\begin{lemma}
		\label{lemma:sg_bounds}
		\begin{enumerate}
			\item For all $v\in  L^\infty(\Omega)$ and $t>0$, we have
				\begin{equation}
		\label{eqn:sg_boundseq1cbis}
		\norm{S(t)v}_\infty \leq \,\norm{v}_\infty,
	\end{equation}
			\begin{equation}
				\label{eqn:sg_boundseq1}
				\norm{\nabla S(t)v}_\infty \leq C_1(\Omega)\bigl(1+t^{-1/2}\bigr)e^{-\lambda t}\norm{v}_\infty.
			\end{equation}
								\item For all $v\in X$ and $t>0$, we have
			\begin{equation}
				\label{eqn:sg_boundseq3}
				\norm{\nabla S (t)v}_\infty \leq C_2(\Omega) e^{-\lambda t}\norm{\nabla v}_\infty,
			\end{equation}
						\begin{equation}
						\label{eqn:sg_boundseq31}
				\norm{S (t)v}_X\leq C_3(\Omega) \norm{v}_X.
			\end{equation}
				\item Let $m\in(n,\infty)$. For all $v\in W^{1,m}(\Omega)$ and $t>0$, we have 
			\begin{equation}
				\label{eqn:sg_boundseq4}
				\norm{\nabla S(t)v}_\infty \leq C_4(\Omega,m)\bigl(1+t^{-1/2}\bigr)e^{-\lambda t} \norm{\nabla v}_m.
			\end{equation}
		\end{enumerate}
	\end{lemma}

\begin{proof}[Proof of Proposition~\ref{prop:finalgen3}]
		Assume that $\eps=\norm{\nabla u_0}_{\infty}\le \eps_0$, where $\eps_0=\eps_0(\Omega,F)>0$ will be chosen below.
		\smallskip
		
			{\bf Step 1.} {\it Trapping region and exponential decay of the gradient.}
		We first claim that
			\begin{equation}
		\label{eqn:convexpeq2a}
		T=\infty\quad\hbox{and}\quad \norm{\nabla u(t)}_\infty\le C\eps e^{-\lambda t},\quad t\ge 0.
			\end{equation}
			Here and in the rest of the proof, $C$ denotes a generic positive constant depending only on $\Omega$.
		Using the representation formula
			\begin{equation}
		\label{eqn:convexpeq2}
		u(t)=S(t)u_0+\int_0^t S(t-s)F(\nabla u(s))ds
			\end{equation}
	and Lemma~\ref{lemma:sg_bounds}, we have
	\begin{equation}
		\label{eqn:convexpeq2b}
			\begin{aligned}
			\norm{\nabla u(t)}_\infty & \leq \norm{\nabla S(t)u_0}_\infty + \int_0^t \norm{\nabla S(t-s)F(\nabla u(s))}_\infty ds \\
			& \le C_1e^{-\lambda t}\norm{\nabla u_0}_{\infty}
			+C_1\int_0^t e^{-\lambda(t-s)}\bigl(1+(t-s)^{-1/2}\bigr)\norm{F(\nabla u(s))}_\infty ds.
					\end{aligned}
	\end{equation}
Assume $\eps_0<(2C_1)^{-1}\delta$ (where $\delta$ is as in \eqref{eqn:ass3}) and set 
$$N(\tau)\defeq \sup_{s\in (0,\tau)} e^{\lambda s}\norm{\nabla u(s)}_\infty,\ \tau\in (0,T)
\quad\hbox{and}\quad
J=\bigl\{\tau\in (0,T),\ N(\tau)\le 2 C_1\eps\bigr\}.$$
	By continuity, it follows from \eqref{eqn:convexpeq2b} that $J\ne\emptyset$.
	Let $T_0=\sup J$ and assume for contradiction that $T_0<T$, hence $N(T_0)=2C_1\eps$ by continuity. 
	For all $t\in(0,T_0]$, we then have 
	$$\norm{\nabla u(t)}_\infty\le N(T_0)e^{-\lambda t}\le 2C_1\eps e^{-\lambda t}\le \delta$$
	 hence, by \eqref{eqn:ass3},
	\begin{equation}
		\label{eqn:convexpeq2c}
		|F(\nabla u(t))|\le |\nabla u(t)|^p\le N^p(T_0)e^{-p\lambda t}\le (2C_1\eps)^p  e^{-p\lambda t},\quad 0\le t<T_0.
			\end{equation}
It follows from \eqref{eqn:convexpeq2b} that
		$$e^{\lambda t}\norm{\nabla u(t)}_\infty \le C_1\eps
					+C_1(2C_1\eps)^p\int_0^t \bigl(1+(t-s)^{-1/2}\bigr)e^{-\lambda (p-1)s} ds.
$$
On the other hand, there exists $C_2=C_2(p,\Omega)>0$ such that, for all $t\ge 0$,
$$\int_0^t \bigl(1+(t-s)^{-1/2}\bigr)e^{-\lambda (p-1)s} ds
\le 2\int_0^{(t-1)}  e^{-\lambda (p-1)s}   ds+2\int_{(t-1)}^t (t-s)^{-1/2} ds\le C_2.$$
Therefore,
$$2C_1\eps=N(T_0)\le C_1\eps + C_1C_2 (2C_1\eps)^p
=C_1\eps\bigl(1+ (2C_1)^pC_2 \eps^{p-1}\bigr)$$
hence $1\le (2C_1)^pC_2 \eps^{p-1}$.
Further assuming $\eps_0<((2C_1)^pC_2)^{-1/(p-1)}$, we reach a contradiction.
Consequently, $T_0=T$ and, recalling \eqref{eqn:cont} and \eqref{eqn:bounduu}, we deduce \eqref{eqn:convexpeq2a}.

\smallskip
		{\bf Step 2.} {\it Definition of the constant $c$ and $L^\infty$ convergence.}
We set
$$
	c \defeq \fint_\Omega u_0 + \int_0^\infty \fint_\Omega F(\nabla u(t))dt.
$$
Note that $c$ is well defined, since \eqref{eqn:convexpeq2a} guarantees the absolute convergence of the second integral.
For all $t\ge 0$, integrating \eqref{eqn:HJ} in space and using \eqref{eqn:convexpeq2c}, 
we obtain 
$$
		\Bigl|c-\fint_\Omega u(t)\Bigr|=\frac{1}{|\Omega|}\Bigl|\int_t^\infty \int_\Omega F(\nabla u(s))ds\Bigr|
		\le (2C_1\eps)^p\int_t^\infty  e^{-p\lambda s}ds\le  (p\lambda)^{-1}(2C_1\eps)^p e^{-p\lambda t}.
$$
Since, on the other hand, 
$$
	\norm{u(t)-\fint u(t)}_\infty\leq  C \norm{\nabla u(t)}_\infty \leq 2CC_1\eps e^{-\lambda t},
$$
we deduce that
$$
\begin{aligned}
	\norm{u(t)-c}_\infty
	&\leq \norm{c-\fint u(t)}_\infty+\norm{u(t)-\fint u(t)}_\infty 
	\leq  (p\lambda)^{-1}(2C_1\eps)^p e^{-p\lambda t}+2CC_1\eps e^{-\lambda t} 
\leq C\eps e^{-\lambda t}.
\end{aligned}
$$
This completes the proof.
\end{proof}

\section{Uniformity of the constant $C$ and proof of Theorem \ref{thm:finalgen}}
\label{sec:conv2}

Based on the previous sections, we can give the:
\smallskip

\begin{proof}[Proof of Theorem \ref{thm:finalgen}(i)]
It is a direct consequence of Propositions~\ref{prop:finalgen}, \ref{prop:finalgen2}, \ref{prop:finalgen3}.
\end{proof}

The goal of this section is then to show the uniformity of the constant $C$ in \eqref{eqn:exprate}.
We shall prove the following proposition, which implies Theorem \ref{thm:finalgen}(ii).

	\begin{proposition}
	\label{prop:finalgen4}
	Consider problem \eqref{eqn:HJ} under assumptions \eqref{eqn:ass0}--\eqref{eqn:ass5}, with $u_0\in X$.
	Then the constant $C$ in  \eqref{eqn:exprate}  is uniform for $u_0$ in bounded sets of $X$.
\end{proposition}

\begin{proof}
Let $\kappa>0$ and let $u_0\in X$ be such that $\norm{u_0}_{X}\le\kappa$.
In this proof, we denote by $D_1, D_2, \dots$ various positive constants which depend on $u_0$ through $\kappa$ only
(and may depend on $\Omega,F$).
By the small time estimate \eqref{eqn:mild2}, there exist $\varepsilon=\varepsilon(\kappa)>0$ such that 
	\begin{equation}
		\label{boundeps1}
		\sup_{t\in[0,\eps]} \|u(t)\|_X\le D_1.
	\end{equation}
 By the argument in the first part of the proof of Proposition~\ref{prop:finalgen2}, it follows that 
 $\|u_t(\eps)\|_\infty\le D_2$.
Using estimate \eqref{barh2} with the choice $t_0=\eps$ and \eqref{boundeps1}, we deduce that 
	\begin{equation}
		\label{boundeps2}
		\sup_{t>0} \|u(t)\|_X\le D_3.
			\end{equation}

On the other hand, integrating \eqref{eqn:HJ} in space and using \eqref{eqn:bounduu} and assumption \eqref{eqn:ass5}, we have
$$ \begin{aligned}
	c_0\int_0^t \int_\Omega |\nabla u|^q 
	&\le  \int_0^t \int_\Omega F(\nabla u) =  \int_\Omega u(t) - \int_\Omega u_0 \\
&\le |\Omega|(\sup_\Omega u(t) - \inf_\Omega u_0)
\le |\Omega|(\sup_\Omega u_0 - \inf_\Omega u_0)\le 2|\Omega|\kappa.
		\end{aligned}		$$
We deduce that, for all $t>0$,
$\inf_{s\in(0,t)} \|\nabla u(s)\|^q_{L^q(\Omega)}\le \frac{2|\Omega|\kappa}{c_0 t}$.
 For $m\in(q,\infty)$, 
interpolating with \eqref{boundeps2} and denoting by $\|\cdot\|_m$ the norm in $L^m(\Omega)$, we obtain
	\begin{equation}
		\label{boundeps3}
		\inf_{s\in(0,t)} \|\nabla u(s)\|_{L^m(\Omega)}\le \Bigl(\frac{2|\Omega|\kappa}{c_0 t}\Bigr)^{1/m}D_3^{1-q/m},
		\quad t>0.
					\end{equation}

We next proceed to apply a smoothing argument to improve the $L^m$ norm to $L^\infty$ in \eqref{boundeps3}. 
First note that \eqref{boundeps2} and \eqref{eqn:ass0} imply
	\begin{equation}
		\label{boundeps2b}
		|F(\nabla u(s))|\le D_4|\nabla u(s)|,\quad s\ge 0.
						\end{equation}
Fix some $m\in (\max(n, q),\infty)$ and let $t_1>0$ and $t\in(0,1]$. By the variation of constants formula \eqref{eqn:convexpeq2}
and the semigroup estimates \eqref{eqn:sg_boundseq1} and \eqref{eqn:sg_boundseq4}, 
it follows that
$$
			\begin{aligned}
			\norm{\nabla u(t_1+t)}_\infty & \leq \norm{\nabla S(t)u(t_1)}_\infty + \int_0^t \norm{\nabla S(t-s)F(\nabla u(t_1+s))}_\infty ds \\
			& \le Ct^{- 1/2}\norm{\nabla u(t_1)}_m
			+C_1D_4\int_0^t (t-s)^{-1/2}\|\nabla u(t_1+s)\|_\infty ds,
					\end{aligned}
$$
where $C_1=C_1(\Omega,m)$,
hence
$$		t^{ 1/2}\norm{\nabla u(t_1+t)}_\infty  \le C\norm{\nabla u(t_1)}_m
			+C_1D_4t^{1/2} \int_0^t (t-s)^{-1/2}\|\nabla u(t_1+s)\|_\infty ds.$$
Let $t_2\in(0,1]$. Taking supremum over $t\in(0,t_2]$, it follows that $\Lambda(t_2):=\sup_{{s}\in(0,t_2)} s^{1/2}\|\nabla u(t_1+s)\|_\infty$ satisfies
$$			\Lambda(t_2) \leq C\norm{\nabla u(t_1)}_m+C_1D_4\Lambda(t_2) \sup_{t\in(0,t_2]} t^{ 1/2}\int_0^t (t-s)^{-1/2}s^{-{ 1/2}}ds
 \leq C\norm{\nabla u(t_1)}_m+C_1D_4t_2^{1/2}\Lambda(t_2).$$
 Taking $t_2=\min\bigl(1,(2C_1D_4)^{-2}\bigr)$, we obtain
$$t_2^{ 1/2}\norm{\nabla u(t_1+t_2)}_\infty \le \Lambda(t_2) \le
 2C\norm{\nabla u(t_1)}_m$$
hence, combining with \eqref{boundeps3},
 $$\inf_{s\in(t_2,t_1+t_2)} \|\nabla u(s)\|_\infty 
 \le 2Ct_2^{-{1/2}}\Bigl(\frac{2|\Omega|\kappa}{c_0 t_1}\Bigr)^{1/m}D_3^{1-q/m}
 \le D_5t_1^{-1/m},\quad t_1>0.$$

 Now, we choose 
$t_1=(D_5/\eps_0)^m$, where $\eps_0=\eps_0(F,\Omega)$ is given by Proposition~\ref{prop:finalgen3}.
 There must exist $t_3\in (t_2,t_1+t_2)$ such that $\|\nabla u(t_3)\|_\infty \le\eps_0$. Then $t_3\leq 1+(D_5/\varepsilon_0)^m\eqdef D_6$.
Proposition~\ref{prop:finalgen3}, applied after shifting the time origin, then guarantees the existence of $c\in\R$,
given by \eqref{defc},
such that
	\begin{equation}
		\label{boundeps2c}
		\norm{u(t)-c}_X\leq C(\Omega)\|\nabla u(t_3)\|_\infty e^{-\lambda(t-t_3)}
\leq C(\Omega)\eps_0e^{\lambda D_6} e^{-\lambda t}\le D_7 e^{-\lambda t}, \quad t\geq t_3.
						\end{equation}
					In particular, 
						\begin{equation}
	\label{boundeps2cbis}
	\norm{\nabla u(t)}_\infty\leq  D_7 e^{-\lambda t}, \quad t\geq t_3.
\end{equation}					
 To treat the remaining range $t\le t_3$, we write
	\begin{equation}
		\label{boundeps2d}
		\begin{aligned}
			\norm{u(t)-c}_X
			&\le \abs{c}+\norm{u(t)}_X \myleq{\eqref{boundeps2}}
			\Bigl|\int_0^\infty \fint_\Omega F(\nabla u(s))ds\Bigr|+\abs{\fint_{\Omega}u_0}+D_3\\
			&\leq\int_0^{t_3} \fint_\Omega |F(\nabla u(s))|ds
			+\int_{t_3}^\infty \fint_\Omega |F(\nabla u(s))|ds+\kappa+D_3.
		\end{aligned}
	\end{equation}
Now, using \eqref{boundeps2b}, \eqref{boundeps2} and the fact that $t_3\leq D_6$, we have
\begin{equation}
	\label{boundeps2e}
	\int_0^{t_3} \fint_\Omega |F(\nabla u(s))|ds\myleq{\eqref{boundeps2b}}\int_0^{t_3} \fint_\Omega D_4 \abs{\nabla u(s)}ds \myleq{\eqref{boundeps2}}\int_0^{t_3} D_4D_3\leq D_6D_4D_3.
\end{equation}
Furthermore, using \eqref{boundeps2b} and \eqref{boundeps2cbis}, we get
\begin{equation}
	\label{boundeps2f}
	\int_{t_3}^\infty \fint_\Omega |F(\nabla u(s))|ds\myleq{\eqref{boundeps2b}}\int_{t_3}^\infty \fint_\Omega D_4 \abs{\nabla u(s)}ds\myleq{\eqref{boundeps2cbis}}\int_{t_3}^\infty D_4 D_7 e^{-\lambda s} \leq \frac{D_4D_7}{\lambda}.
\end{equation}
Finally, combining \eqref{boundeps2d}-\eqref{boundeps2f}, we obtain
$$
		\norm{u(t)-c}_X
		\leq D_6D_4D_3 + \frac{D_4D_7}{\lambda}+\kappa+D_3=:
		D_8\le D_8e^{\lambda D_6}e^{-\lambda t}=D_9e^{-\lambda t},\qquad t\le t_3.
$$
 The proof is complete.
\end{proof}

\section*{Acknowledgments}
 The authors thank the referee for careful reading of the manuscript, and Said Benachour for useful suggestions.
 This work was done during a visit of the first author at LAGA, Universit\'e Sorbonne Paris Nord.
He thanks this institution for its support.

\appendix

\section{Appendix: Local existence-uniqueness-regularity and comparison principle}

{We first give the proof of  Lemma~\ref{lemma:sg_bounds}.}

\medskip

	\begin{proof}[Proof of  Lemma~\ref{lemma:sg_bounds}]
	Property \eqref{eqn:sg_boundseq1cbis} follows from the maximum principle.
	We also have
\begin{equation}
\label{eqn:sg_boundseq1c}
				\norm{\nabla S(t)v}_\infty \leq C\,\frac{\norm{v}_\infty}{\sqrt{t}},\quad t\in(0,1],\quad v\in L^\infty(\Omega),
	\end{equation}
	as a consequence of well-known heat kernel estimates (see, e.g., \cite[Theorem~2.2]{Mo}).
	Here and in the rest of the proof, $C$ denotes a positive constant depending only on $\Omega$.
Now, we claim that
\begin{equation}
							\label{eqn:sg_boundseq1b}
				\norm{\nabla S(t)v}_\infty \leq C\norm{\nabla v}_\infty,\quad t\in(0,1],\quad v\in X.
\end{equation}
Although this is probably known, we could not find a proof in the literature covering the case $X=W^{1,\infty}(\Omega)$.
However \eqref{eqn:sg_boundseq1b} can be shown as a consequence of the Bernstein arguments used in 
steps~2 and 3 of the proof of Proposition~\ref{prop:finalgen}.
Namely, denote $w(t)=S(t)v$, let $\psi\in C^2(\overline\Omega)$ be given by \eqref{AuxRobin},
which satisfies $\min_{\overline\Omega}\psi>0$, and set $h=\psi|\nabla w|^2$. 
Recalling that $w\in C([0,\infty);H^1(\Omega))\cap C^{2,1}(\overline\Omega\times (0,\infty))\cap C^\infty(\Omega\times (0,\infty))$,
we have 
\be{regulh}
h\in C([0,\infty);L^2(\Omega))\cap C^{2,1}(\Omega\times (0,\infty)),\quad
\nabla h\in C(\overline\Omega\times (0,\infty)).
\ee
By \eqref{eqn:globalexeq8} with the choice $F=0$, $h$ satisfies 
$h_t-\Lap h \le C(|\nabla h|+h)$ in $\Omega\times (0,\infty)$, 
so that $z:=e^{-Ct}h-\|h(0)\|_\infty$ solves
$$z_t-\Lap z \le C|\nabla z|\quad\hbox{in $\Omega\times (0,1]$}$$
with $z(0)\le 0$ and, moreover, $\frac{\partial z}{\partial \nu}\le 0$ on $\partial\Omega\times (0,1]$ owing to \eqref{hNeumann}.
In view of the regularity \eqref{regulh}, we may apply the maximum principle in Proposition~\ref{thm:comp0}
to infer that $z\le 0$, and this implies~\eqref{eqn:sg_boundseq1b}.

		Next assume $\fint_\Omega v=0$. Then
		$$
			\norm{S(t)v}_2\leq e^{-\lambda t}\norm{v}_2, \quad t>0.
		$$
		By the regularizing effect of the semigroup, it follows that
		$$
			\norm{S(t)v}_\infty\leq C\norm{S(t-1)v}_2 \leq Ce^{-\lambda (t-1)}\norm{v}_2\leq Ce^{-\lambda t}\norm{v}_\infty,\quad t\ge 1, 
		$$
		and in fact we can extend this result to any $t>0$ due to \eqref{eqn:sg_boundseq1cbis},
		\begin{equation}
			\label{eqn:sg_boundseq3bbis}
			\norm{S(t)v}_\infty\leq Ce^{-\lambda t}\norm{v}_\infty,\quad t\ge 0.
		\end{equation}
		Using \eqref{eqn:sg_boundseq1c}, we obtain 
	$$
			\norm{\nabla S(t)v}_\infty = \norm{\nabla S(1) \bigl(S(t-1) v\bigr)}_\infty\myleq{\eqref{eqn:sg_boundseq1c}} 
			C\norm{ S(t-1) v}_\infty\myleq{\eqref{eqn:sg_boundseq3bbis}} Ce^{-\lambda t}\norm{v}_\infty,\quad t\ge 1
$$
		and then, using \eqref{eqn:sg_boundseq1c} again,
				\begin{equation}
				\label{eqn:sg_boundseq3b}
			\norm{\nabla S(t)v}_\infty \leq C(1+t^{-1/2})e^{-\lambda t}\norm{v}_\infty,\quad t>0.
		\end{equation}

		Now, in the general case, set $\bar v=\fint_\Omega v$. Using $S(t)\bar{v}=\bar{v}$, $\nabla S(t)\bar v=\nabla \bar v=0$ 
		and applying \eqref{eqn:sg_boundseq3b} to $v-\bar v$, we obtain
					\begin{equation}
				\label{eqn:sg_boundseq3b1}
\bigl\|\nabla S(t)v\bigr\|_\infty=\bigl\|\nabla S(t)(v-\bar v)\bigr\|_\infty\leq  C(1+t^{-1/2})e^{-\lambda t}\norm{v-\bar v}_\infty,
		\quad t>0.
			\end{equation}
		Since $\abs{\bar{v}}\leq \norm{v}_\infty$ and therefore, $\norm{v-\bar v}_\infty\le 2\norm{v}_\infty$,  inequality \eqref{eqn:sg_boundseq1} follows.
		As for \eqref{eqn:sg_boundseq3}, it is a consequence of 
		\eqref{eqn:sg_boundseq1b}, \eqref{eqn:sg_boundseq3b1} and $\norm{v-\bar v}_\infty\le C\norm{\nabla v}_\infty$.
		Property \eqref{eqn:sg_boundseq31} is a consequence of \eqref{eqn:sg_boundseq3} and \eqref{eqn:sg_boundseq1cbis}.
		
		Finally, for $v\in W^{1,m}(\Omega)$, since  $\norm{v-\bar v}_\infty\le C\norm{v-\bar v}_{W^{1,m}(\Omega)}\le C\norm{\nabla v}_{L^m(\Omega)}$
		by the Morrey-Sobolev and Poincar\'e-Wirtinger inequalities (see, e.g., \cite[Theorem 12.23]{Le}), property 
		\eqref{eqn:sg_boundseq4} follows from \eqref{eqn:sg_boundseq3b1}.
		\end{proof}

We give a precise local existence-uniqueness statement for problem \eqref{eqn:HJ} with initial data $u_0\in X$.

	\begin{theorem}
	Assume \eqref{eqn:ass0}. Let $\kappa>0$ and let $u_0\in X$ be such that $\|u_0\|_X\le \kappa$.
		\smallskip
		
		(i) There exists $\varepsilon=\varepsilon(\kappa,F)>0$, and a unique solution  $u\in L^\infty(0,\varepsilon;X)$ of
		\begin{equation}
			\label{eqn:mild}
			u(t)=S(t)u_0 + \int_0^t S(t-s)F(\nabla u(s)),\quad 0<t<\eps.
		\end{equation}	
		Moreover, we have 
				\begin{equation}
			\label{eqn:mildcont}
			\lim_{t\to 0}\|u(t)-S(t)u_0\|_X=0.
					\end{equation}	
		 In addition, we may choose $\eps$ so that 
		\begin{equation}
			\label{eqn:mild2}
			\sup_{t\in[0,\eps]} \|u(t)\|_X\le  C(\Omega)\kappa.
					\end{equation}	
	
	 (ii) Let $\tau>0$ and $v\in L^\infty(0,\tau;X)$ be a solution of \eqref{eqn:mild} on $(0,\tau)$.
	Then $v$ is unique, $v\in C^{2,1}(\adh{\Omega}\times (0,\tau))$ and $v$ is a classical solution of \eqref{eqn:HJ}.

	\smallskip

	(iii) There exists $T>0$ and a unique maximal solution  $u\in L^\infty(0,T ;X)$ of \eqref{eqn:mild}.
	If $T<\infty$, then $\lim_{t\to T}\norm{u(t)}_{X}=\infty$.
		\end{theorem}

Although the result is essentially standard, we provide a proof for the convenience of the readers.

\smallskip

	\begin{proof}
	Let $Y_\tau:= L^\infty((0,\tau);X)$, 
with norm $\|v\|_{Y_\tau}:=\esssup_{t\in (0,\tau)}\|v(t)\|_X$.
For given $\tau\in(0,1]$ and $M>0$, the closed ball $B_{\tau,M}=\{v\in Y_\tau:\ \|v\|_{Y_\tau}\le M\}$
is a complete metric space.
For $v\in Y_\tau$, we define
$$\mathcal{T}(v)(t):=S(t)u_0+\int_0^t S(t-s)F(\nabla v(s))\, ds,\quad 0< t<\tau.$$
We set
$$M:=2C_3\|u_0\|_X,\quad K=K_M:=\sup_{|z|\le M} |F(z)|
\quad\hbox{and}\quad L=L_M:=\sup_{|z|\le M} |\nabla F(z)|,$$
where $C_3$ is the constant from \eqref{eqn:sg_boundseq31}.
By \eqref{eqn:sg_boundseq1cbis}, \eqref{eqn:sg_boundseq1}, there exists $C_4=C_4(\Omega)>0$ such that
				\begin{equation} \label{eqn:sg_boundseq1cter}
		\norm{S(t)\phi}_X \leq C_4t^{-1/2}\norm{\phi}_\infty,\quad \phi\in  L^\infty(\Omega),\quad 0<t\le 1.
			\end{equation}
For any $v,w\in B_{\tau,M}$, using estimates \eqref{eqn:sg_boundseq1cter}, \eqref{eqn:sg_boundseq31}, we have
$$\begin{aligned}
\|\mathcal{T}(v)(t)\|_X
&\le \|S(t)u_0\|_X+\int_0^t \|S(t-s)(F(\nabla v(s)))\|_X\, ds\\
&\le C_3\|u_0\|_X+C_4\int_0^t (t-s)^{-1/2}\|F(\nabla v(s)))\|_\infty\, ds 
\le C_3\|u_0\|_X+2C_4K\tau^{1/2}
\end{aligned}
$$
and
$$\begin{aligned}
\|\mathcal{T}(v)(t)-\mathcal{T}(w)(t)\|_X
&\le \int_0^t \|S(t-s)(F(\nabla v(s)))-F(\nabla w(s))))\|_X\, ds \\
&\le C_4\int_0^t (t-s)^{-1/2}\|F(\nabla v(s)))-F(\nabla w(s)))\|_\infty\, ds \\
&\le C_4L\int_0^t (t-s)^{-1/2}\|v-w\|_{Y_\tau}, \\
\end{aligned}
$$
hence
$$\|\mathcal{T}(v)(t)-\mathcal{T}(w)(t)\|_{Y_\tau}\le 2C_4L\tau^{1/2} \|v-w\|_{Y_\tau}.$$
Choosing $\eps=\tau\in(0,1]$ small, so that $2C_4\tau^{1/2}\le \min\bigl((2L)^{-1},C_3\kappa/K\bigr)$, it follows that 
$\mathcal{T}: B_{\tau,M}\to B_{\tau,M}$ is a contraction mapping. Consequently, 
by the Banach fixed point theorem, $\mathcal{T}$
admits a unique fixed point $u$ in $B_{\tau,M}$.
In addition, \eqref{eqn:mild2} is true.
On the other hand, we have
$$\|u(t)-S(t)u_0\|_X\le\int_0^t \|S(t-s)(F(\nabla u(s)))\|_X\, ds
\le C_4\int_0^t (t-s)^{-1/2}\|F(\nabla u(s))\|_\infty\, ds
\le C_4 K t^{1/2},$$
hence \eqref{eqn:mildcont}.

		\smallskip

	(ii) The regularity statement follows, after multiplying the solution by a cut-off vanishing near $t=0$,  
	from $L^p$ and Schauder existence theory 
and uniqueness of weak solutions for linear parabolic problems; 
	see \cite[p.93, first paragraph]{QS} for details on a similar case
	and cf.~also the argument at the beginning of the proof of Proposition~\ref{prop:finalgen2}.
	
		The uniqueness is then a consequence of the regularity of $v$ and of the comparison principle in Proposition~\ref{thm:comp},
		noting that a mild solution $v\in L^\infty(0,\tau;X)$ must satisfy $v\in C([0,\tau);L^2(\Omega))$.

				\smallskip	
	(iii) This follows from a standard continuation argument.
		\end{proof}

We have the following  maximum and comparison principles.
We hereafter denote $Q_\tau=\Omega\times(0,\tau)$, $S_\tau=\partial\Omega\times(0,\tau)$.

\begin{proposition}
	\label{thm:comp0}
	Let $\Omega\subset \Rn$ be a bounded $C^2$ domain. Let $\tau>0$ and $w\in C([0,\tau],L^2(\Omega))$
	be such that $w_t,D^2w\in L^2_{loc}(Q_\tau)$ and 
	\be{nablawC}
	\nabla w\in C(\overline\Omega\times(0,\tau)).
	\ee
	Assume that, for some $K\ge 0$,
	\be{nablawC2}
		\left\{
		\begin{aligned}
			\frac{\partial w}{\partial t}-\Lap w &\le K|\nabla w|&&\hbox{a.e.~in $Q_\tau\cap\{w>0\}$,}\\
			w(0) & \le 0&& \hbox{in $\Omega$,} \\
			 \frac{\partial w}{\partial\nu}& \le 0 &&\hbox{in $S_\tau$.}
		\end{aligned}
	\right.
\ee
Then $w\le 0$ in $Q_\tau$.
\end{proposition}

\begin{proposition}
	\label{thm:comp}
	Let $\Omega\subset \Rn$ be a bounded $C^2$ domain. Let $\tau>0$ and $u,v\in  C^{2,1}(\adh{\Omega}\times(0,\tau])\cap C([0,\tau],L^2(\Omega))\cap  L^\infty(0,\tau;X)$. 
	Assume
$$
		\left\{
		\begin{aligned}
			\frac{\partial u}{\partial t}-\Lap u -  F(\nabla u) & \leq \frac{\partial v}{\partial t}-\Lap v -  F(\nabla v)\qquad && \hbox{in $Q_\tau$,} \\
			u(0) & \leq v(0) \qquad &&  \hbox{in $\Omega$,}\\
			 \frac{\partial u}{\partial\nu}& \leq \frac{\partial v}{\partial\nu} \qquad &&
			 \hbox{in $S_\tau$,}
		\end{aligned}
	\right.
$$
	with $F:\Rn \longrightarrow \R$ of class $C^1$. Then, either $u\equiv v$ or 
$$
		u<v \qquad \text{in} \ \ \adh{\Omega}\times(0,\tau].
$$
\end{proposition}

For Proposition~\ref{thm:comp} we refer to, e.g., \cite[Proposition 52.7]{QS}.
As for Proposition~\ref{thm:comp0}, since $w$ is not smooth up to the boundary, for convenience we give a proof 
(expanding arguments from, e.g., \cite[Theorem~52.8]{QS} and \cite[Remarks 52.9 and 52.11(a)]{QS}).

\smallskip

	\begin{proof}[Proof of Proposition~\ref{thm:comp0}]
{\bf Step 1.} {\it Domain approximation and uniform trace inequality.}
Since $w$ is not smooth up to the boundary, we shall use a suitable approximation of the domain $\Omega$.
Set $\Omega_\delta=\Omega\cap\{{\rm dist}(x,\partial\Omega)>\delta\}$,
$\omega_\delta=\Omega\cap\{{\rm dist}(x,\partial\Omega)<\delta\}$ and 
$d(x)={\rm dist}(x,\partial\Omega)$.
It is well known that, for $\delta_0>0$ sufficiently small, we have
\be{tracedelta1}
d\in C^2(\overline\omega_{\delta_0}),\quad \nu=-\nabla d\ \hbox{ on $\partial\Omega$,}
\ee
\be{tracedelta2}
\hbox{$\omega_\delta$ and $\Omega_\delta$ are $C^2$-smooth for each $\delta\in(0,\delta_0]$.}
\ee
Moreover, reducing $\delta_0$ if necessary, it is easy to see that, for all $\delta\in(0,\delta_0]$,
the function $d_\delta(x):={\rm dist}(x,\partial\Omega_\delta)$ verifies
$d_\delta(x)=d(x)-\delta$ in $\Omega_\delta\cap\omega_{\delta_0}$.
Therefore, the normal vector field $\nu_\delta$ of $\partial\Omega_\delta$ satisfies 
\be{tracedelta3}
\nu_\delta=-\nabla d_\delta=-\nabla d\quad\hbox{on $\partial\Omega_\delta$,\quad $0<\delta<\delta_0$.}
\ee

We next claim that we have the uniform trace inequality 
\be{tracedelta}
\|v\|_{L^2(\partial\Omega_\delta)}
\leq C_1\|v\|_{H^1(\Omega_\delta)},\quad v\in H^1(\Omega_\delta), \quad 0<\delta<\delta_0,
\ee
with $C_1=C_1(\Omega)>0$ independent of $\delta$.
Indeed, by \eqref{tracedelta1} and \eqref{tracedelta2} for $\omega_{\delta_0}$,
 we may fix an extension $\tilde d\in C^2(\overline\Omega)$ such that
$\tilde d=d$ in $\overline\omega_{\delta_0}$. By \eqref{tracedelta2} for $\Omega_\delta$
we may apply the divergence theorem and, using \eqref{tracedelta3}, we get
$$\begin{aligned}
\|v\|^2_{L^2(\partial\Omega_\delta)} 
&=-\int_{\partial\Omega_\delta}  v^2\nabla \tilde d\cdot\nu_\delta\,d\sigma_\delta
=-\int_{\Omega_\delta} \nabla\cdot(v^2\nabla \tilde d)\,dx
=-\int_{\Omega_\delta} (v^2\Delta \tilde d+2v\nabla v\cdot\nabla \tilde d)\,dx \\
&\le \int_{\Omega_\delta} |\nabla v|^2\,dx
+\sup_\Omega\,\bigl(-\Delta \tilde d+|\nabla \tilde d|^2\bigr)\int_{\Omega_\delta} v^2\,dx
\le C(\Omega)\|v\|^2_{H^1(\Omega_\delta)}.
\end{aligned}$$

{\bf Step 2.} {\it Stampacchia argument.}
Our assumptions imply
$w_+\in C([0,\tau),L^2(\Omega_\delta))\cap H^1_{loc}((0,\tau),$ $L^2(\Omega_\delta))$
and, for a.e.~$t\in (0,\tau)$, $w_+(t)\in H^1(\Omega_\delta)$. 
For a.e.~$t\in (0,\tau)$, multiplying the PDE in \eqref{nablawC2} by $w_+$, using $\Delta w(\cdot,t) \in  
L^2(\Omega_\delta)$ and $\nabla(w_+)(\cdot,t)=\chi_{\{w>0\}}\nabla w(\cdot,t)$
and integrating by parts, we get
$$
{1\over 2}{d\over dt}\int_{\Omega_\delta} (w_+)^2(t)\,dx
\leq \int_{\partial\Omega_\delta} w_+(\nabla w(x,t)\cdot \nu_\delta)\,d\sigma_\delta -\int_{\Omega_\delta} |\nabla(w_+)|^2\,dx+K\int_{\Omega_\delta} |\nabla w|w_+\,dx.$$
Fix $\tau_1<\tau$, $t_0\in(0,\tau_1)$ and $\eps>0$. 
By \eqref{nablawC},
there exists $\delta_1\in(0,\delta_0]$ such that, for all $\delta\in(0,\delta_1]$,
$$\sup_{(x,t)\in\partial\Omega_\delta\times(t_0,\tau_1)} \bigl|\nabla w(x,t)\cdot \nabla d(x)-\nabla w(P(x),t)\cdot(\nabla d)(P(x))\bigr|
\le \eps,$$
where $P$ is the projection onto $\partial\Omega$.
By \eqref{tracedelta1}, \eqref{tracedelta3} and the boundary conditions in \eqref{nablawC2}, it follows that
$$\sup_{\partial\Omega_\delta\times(t_0,\tau_1)}\nabla w\cdot \nu_\delta\le \eps$$
and, using~\eqref{tracedelta}, we deduce that
$$\int_{\partial\Omega_\delta} w_+(\nabla w\cdot \nu_\delta)\,d\sigma_\delta
\le  \eps\int_{\partial\Omega_\delta} w_+\,d\sigma_\delta
\le  \eps\int_{\partial\Omega_\delta} \bigl(1+w_+^2\bigr)\,d\sigma_\delta
\myleq{\eqref{tracedelta}}  C_2\eps\Bigl(1+\int_{\Omega_\delta} (w_+^2+|\nabla w_+|^2)\,dx\Bigr),$$
 for all $t\in(t_0,\tau_1)$,
 with $C_2=C_2(\Omega)>0$ independent of $\delta$. 
 Assuming $\eps<1/(2C_2)$, it follows that
 \begin{equation}
 	\label{eqn:compeq1}
\begin{aligned}
{1\over 2}{d\over dt}\int_{\Omega_\delta} (w_+)^2(t)\,dx
&\leq C_2\eps+\Bigl(C_2\eps-1+\frac12\Bigr)\int_\Omega |\nabla (w_+)|^2\,dx
+(C_2\eps+K^2/2)\int_{\Omega_\delta} (w_+)^2\,dx\\
&\leq C_2\eps+\frac{1+K^2}{2}\int_{\Omega_\delta} (w_+)^2\,dx.
\end{aligned}
\end{equation}
Now, as the solution of $y'=a+by$ is given by $y(t)=(y(t_0)+a/b)e^{b(t-t_0)}-a/b$, integrating \eqref{eqn:compeq1}, 
for every $t\in(t_0,\tau_1)$, we obtain
$$
\int_{\Omega_\delta} (w_+)^2(t)\,dx\leq \left(\int_{\Omega_\delta} (w_+)^2(t_0)\,dx + \frac{2C_2\varepsilon}{1+K^2}\right)e^{(1+K^2)(t-t_0)}.
$$
So letting $\delta\to 0$ and then $\eps\to 0$, we get
$\int_\Omega (w_+)^2(t)\,dx\le e^{(1+K^2)(t-t_0)}\int_{\Omega} (w_+)^2(t_0)\,dx$. Finally letting $t_0\to 0$ and
using $w\in C([0,\tau_1],L^2(\Omega))$ and $w(0)\le 0$, we obtain $w\leq 0$ in $Q_{\tau_1}$. As $\tau_1<\tau$ was arbitrary, the conclusion follows.
	\end{proof}

\end{document}